\documentclass{amsart}
\usepackage{amsmath}
\usepackage{graphicx}
\usepackage{latexsym}
\usepackage{pinlabel}
\usepackage{array}

\newcommand{\co}{\mskip0.5mu\colon\thinspace}  

\newtheorem{theorem}{Theorem}[section]
\newtheorem{lemma}[theorem]{Lemma}
\newtheorem{proposition}[theorem]{Proposition}
\newtheorem{corollary}[theorem]{Corollary}

\theoremstyle{remark}
\newtheorem{remark}[theorem]{Remark}

\def\dfn{\emph}

\def\ds{\displaystyle}

\newcommand{\R}{\mathbb{R}}
\newcommand{\N}{\mathbb{N}}

\newcommand{\Z}{\mathbb{Z}}

\newcommand{\bdel}{\widetilde{\Delta}}

\def \bdry {\partial}

\def \x {\times}
\def \eu{{\text{e}}}

\newcommand{\tS}{{\mathbb S^2}}
\newcommand{\D}{{\mathbb D^2}}

\newcommand{\p}{\partial}

\begin{document}

\title[Topology of symplectic $4$-manifolds and Stein fillings]
{Topological complexity of \\ symplectic $4$-manifolds and Stein fillings}

\author[R. \.{I}. Baykur]{R. \.{I}nan\c{c} Baykur}
\address{Max Planck Institut f\"ur Mathematik, Bonn, 53111, Germany \newline
\indent Department of Mathematics, Brandeis University, Waltham, MA 02453, USA}
\email{baykur@mpim-bonn.mpg.de, baykur@brandeis.edu}

\author[J. Van Horn-Morris]{Jeremy Van Horn-Morris}
\address{Department of Mathematics, The University of Arkansas, \newline 
\indent Fayetteville, AR 72703,  USA }
\email{jvhm@uark.edu }

\begin{abstract}
We prove that there exists no a priori bound on the Euler characteristic of a closed symplectic $4$-manifold coming solely from the genus of a compatible Lefschetz pencil on it, nor is there a similar bound for Stein fillings of a contact $3$-manifold coming from the genus of a compatible open book --- except possibly for a few low genera cases. To obtain our results, we produce the first examples of factorizations of a boundary parallel Dehn twist as arbitrarily long products of positive Dehn twists along non-separating curves on a fixed surface with boundary. This solves an open problem posed by Auroux, Smith and Wajnryb, and a more general variant of it raised by Korkmaz, Ozbagci and Stipsicz, independently.
\end{abstract}

\maketitle

\setcounter{secnumdepth}{2}
\setcounter{section}{0}


\section{Introduction} 

Following the ground-breaking works of Donaldson \cite{Donaldson} and Giroux \cite{Gi2}, Lefschetz pencils and open books have become central tools in the study of symplectic \linebreak $4$-manifolds and contact $3$-manifolds. An open
question at the heart of this relationship is whether or not there exists an a priori bound on the topological complexity of a symplectic $4$-manifold, imposed by the genus of a compatible relatively minimal Lefschetz pencil on it. A similar question inquires if such a bound exists for all Stein fillings of a fixed contact $3$-manifold, imposed by the genus and the number of binding components of a compatible open book --- bounding a compatible allowable Lefschetz fibration on each one of these Stein fillings. 

In either case, it is easy to see that there is an upper bound on the first Betti number and a lower bound on the Euler characteristic in terms of the genus and the number of boundary components of a regular fiber. Furthermore, the Euler characteristic (equivalently, the second Betti number) is known to be bounded by above in some low genera cases: The only closed symplectic $4$-manifolds admitting genus one Lefschetz pencils are the blow-ups of the complex projective plane \cite{Kas, Moishezon} and Smith proved that only finitely many characteristic numbers are realized by genus two Lefschetz pencils \cite{SmithGenus2}. On the other hand, Kaloti observed that if a contact $3$-manifold can be supported by a planar open book, Euler characteristics of its Stein fillings constitute a finite set \cite{Kaloti} (also see \cite{Stipsicz2, Etnyre}). Nevertheless, the question on the existence of an upper bound on the Euler characteristic in terms of the topology of a regular fiber remained as a wide open question in both cases, and is the main focus of this article. 

Let $\Gamma_g^n$ denote the orientation-preserving mapping class group of a genus $g$ orientable surface with $n$ boundary components, and assume that $n \geq 1$. A gentle translation of the problems above gives rise to the following question: Is there an a priori upper bound on the length of any factorization of a boundary parallel Dehn twist as a product of positive Dehn twists along homotopically non-trivial curves in $\Gamma_g^n$? This question was raised by Auroux \cite[Question 3]{Auroux}, who attributed it to Smith, and was also discussed by Wajnryb \cite{Wajnryb}. A more general version of the question for \emph{any} element in $\Gamma_g^n$ was stated as an open problem by Ozbagci and Stipsicz \cite[Conjecture 15.3.5]{OS}, who conjectured that such an upper bound exists for any mapping class, by Korkmaz \cite[Problem 2.9]{Korkmaz}, and by Korkmaz and Stipsicz \cite[Problem 7.6]{KorkmazStipsicz}. 

Our main theorem answers all these questions in the negative: 

\begin{theorem} \label{mainthm}
The positive Dehn multitwist along the boundary can be factorized as a product of arbitrarily large number of positive Dehn twists along non-separating curves in $\Gamma_g^2$, provided $g \geq 11$. 
\end{theorem}

\noindent Recall that a Dehn multitwist along the boundary is nothing but a Dehn twist performed along \emph{each} boundary component simultaneously. In particular we obtain the same result for the boundary parallel Dehn twist in $\Gamma_g^1$ via the boundary capping homomorphism from $\Gamma_g^2$ to $\Gamma_g^1$. It therefore follows that there are many other mapping classes attaining the same property: one can for instance take higher powers of the boundary parallel positive Dehn twist $t_\delta$, or its product with any given product of positive Dehn twists along non-separating curves. Not all do though: for example, non-positive powers of the boundary parallel Dehn twist cannot be written as a product of positive Dehn twists to begin with. We should also note that by adding (possibly twisted) one handles to one of the boundary components, we can immediately extend our results to produce mapping classes with arbitrarily long positive Dehn twist factorizations on orientable surfaces with \textit{any} number of boundary components, as well as on non-orientable surfaces.

Going back to our original questions, we have the following result regarding the topological complexity of closed symplectic $4$-manifolds: 

\begin{theorem} \label{cor1}
For each $g \geq 11$, there is a family of relatively minimal genus $g$ Lefschetz pencils $\{ (X_m, f_m) | m \in \N \}$ such that the Euler characteristic of the closed symplectic $4$-manifold $X_m$ is strictly increasing in $m$. Moreover, for any fixed number $l \in N$, we can produce such a family of genus $g=11+ 4l$ Lefschetz pencils, each having the same first Betti number which is strictly increasing in $l$ 
\end{theorem}

\noindent The assumptions on the relative minimality and the existence of base points (equivalently, the existence of a section of square $-1$ for the Lefschetz fibration obtained after blowing-up the base points) rule out two well-known ways of inflating the Euler characteristic: one can blow up along the fibers without destroying the fibration structure or can take non-trivial fiber sums. The relative minimality, by definition, rules out the former, whereas the latter is ruled out by a theorem of Stipsicz \cite{Stipsicz}, also see \cite{Smith}. If there were such an a priori bound on the genus $g$ relatively minimal Lefschetz fibrations with $-1$ sections, it would imply a bound on Euler characteristics of minimal symplectic $4$-manifolds admitting genus $g$ Lefschetz pencils. (In fact, an adjunction argument shows that this holds true even for non-minimal symplectic $4$-manifolds, provided the $4$-manifold is neither rational nor ruled --- as it is the case for our examples above.) 

\newpage
Another result we obtain concerns the topological complexity of Stein fillings: 
 
\begin{theorem} \label{cor2}
For each $g \geq 11$, there are infinitely many closed $3$-manifolds admitting genus $g$ open books with connected binding which bound allowable Lefschetz fibrations over the $2$-disk with arbitrarily large Euler characteristics. Moreover, for each $l \in \N$, we can produce such a family of closed contact $3$-manifolds, where each member is supported by a genus $g=11+4l$ open book bounding allowable Lefschetz fibrations, each having the same first Betti number which is strictly increasing in $l$.   
\end{theorem}

Since allowable Lefschetz fibrations can be equipped with Stein structures inducing the same contact structure as the one induced by the boundary open book \cite{LP, AO}, from Theorem~\ref{cor2} we get infinitely many closed contact $3$-manifolds each of which can be filled by Stein manifolds with arbitrarily large Euler characteristics. Therefore we obtain a new proof of the \textit{``half''} of our main theorem in \cite{BV}, without addressing the unboundedness of the signature this time: namely, we see that there is a new and rather large class of contact $3$-manifolds which admit Stein fillings with arbitrarily large Euler characteristics. Interested readers should compare Corollary~\ref{ALSFthm} below with Theorem~1.1 in \cite{BV}. Nevertheless, this new proof cannot be claimed to be all independent of our previous one, since our close analysis of the monodromies of the allowable Lefschetz fibrations \textit{over positive genera surfaces} in \cite{BV} played a vital role in our discovery of the positive factorizations we have in the current article.

Lastly, we shall add that we are able to improve all three theorems above so as to cover any $g \geq 8$ instead of $g \geq 11$ (but not further; see the discussion in Section~\ref{Sec: Final}). This improvement however comes at an expense of complicating our arguments at a mostly technical level, which we preferred to avoid. We present our proofs for the relatively simpler case of $g \geq 11$ and provide only a brief sketch of how the arguments should be modified to strike the better bound $g \geq 8$ (See Remark~\ref{Genus8}). 

The outline of our paper is as follows: 

In Section~\ref{Sec: Preliminaries} we review the background material relating mapping class group factorizations to Lefschetz pencils and open books, and in turn, to symplectic $4$-manifolds and contact $3$-manifolds. In Section~\ref{Sec: Construction} we prove Theorem~\ref{mainthm}. Our construction of these monodromies is more geometric than algebraic: the desired monodromies are built using mapping classes of a surface which swap certain subsurfaces. These \textit{swap maps} are described as lifts of natural braid maps which admit quasipositive factorizations, so they admit factorizations into positive Dehn twists themselves. Furthermore, our geometric construction allows us to introduce a simple calculus involving various swap maps. We then custom tailor our monodromies out of these blocks on a genus $g \geq 11$ surface with two boundary components to realize commutators of maps on a genus two subsurface, which we can in turn express as arbitrarily long products of positive Dehn twists. Section~\ref{Sec: Applications} is where we derive our results on the topology of closed symplectic $4$-manifolds and contact $3$-manifolds. Section~\ref{Sec: Final} gathers some final comments and questions.

\vspace{0.2in}
\noindent \textit{Acknowledgements.} The first author was partially supported by the NSF grant DMS-0906912. 

\newpage
\section{Preliminaries} \label{Sec: Preliminaries}
 
Here we review the connections between symplectic $4$-manifolds and Stein fillings of contact $3$-manifolds with positive factorizations in the mapping class groups of surfaces --- via Lefschetz fibrations/pencils and open books. While we provide a detailed review of descriptions of all via mapping class group factorizations, we will assume familiarity with the basic notions of symplectic, Stein, or contact structures, for which the reader can for instance turn to \cite{GS}.  

All manifolds in this article are assumed to be compact, smooth and oriented, and all maps are assumed to be smooth. 

\subsection{Mapping class groups and braid groups} \

Let $\Sigma_{g,r}^s$ denote a compact oriented surface of genus $g$ with $s$ boundary components and $r$ marked points in its interior. The \emph{mapping class group} of $\Sigma_{g,r}^s$, denoted by $\Gamma_{g,r}^s$, is the group of isotopy classes of orientation-preserving self-diffeomorphisms of $\Sigma_{g,r}^s$, fixing the points on the boundary and fixing each marked point individually. The isotopies of $\Sigma_{g,r}^s$ are assumed to fix the points on the boundary as well as the marked points. For simplicity, we will often drop $r$ or $s$ from our notation when they are equal to zero.

We denote a positive (right-handed) Dehn twist along an embedded simple closed curve $a$ on $\Sigma_g^s$ by $t_a$, and a negative (left-handed) Dehn twist by $t_a^{-1}$. It is a fundamental fact that $\Gamma_g^s$ is generated by positive and negative Dehn twists. Our focus will be on elements of $\Gamma_{g}^s$ which can be expressed as a product of positive Dehn twists. Any such expression of an element $\Phi \in \Gamma_g^s$ will be called a \emph{positive factorization} of $\Phi$.

Most of the maps and their geometric descriptions we use in this article will be inferred by lifting diffeomorphisms of the unit $2$-disk $\D$ to a compact orientable surface $F$, under a double branched covering. So they arise via a homomorphism from $B_n$, the $n$-stranded braid group on $\D$, to $\Gamma_g^s$ under a fixed identification of $F$ with $\Sigma_g^s$. As usual, we study the $n$-stranded braid group $B_n$ by identifying $\D$ with the unit disk in $\R^2$ and looking at the orientation-preserving self-diffeomorphisms of the latter which preserve $n$ marked points on the $x$-axis (or $y$-axis) \textit{setwise}, up to isotopies. $B_n$ is generated by \textit{half-twists} along embedded arcs in the interior of $\D$ joining any pair of marked points, while avoiding the other marked points. (Further, there are \textit{standard generators}, half-twists along arcs on the $x$-axis joining two consecutive points.) A half-twist is called \textit{positive} if we exchange the marked points in a right-handed manner, or equivalently when the square of a half-twist is a positive Dehn twist along the boundary of a small disk containing the arc in the complement of the other marked points. As above, we will focus on elements of $B_n$, called the \emph{quasipositive braids}, which can be expressed as a product of positive half-twists. We call such a factorization of an element $b \in B_n$, a \emph{quasipositive factorization}.

We will also make use of the \textit{framed braid group} $B_{*n}$, isomorphic to $B_n \x \Z^n$. One can think of this as an extension of $B_n$ where we keep track of the twisting around each marked point. To do this, one picks a direction vector at every marked point and as before, the elements of $B_{*n}$ are thought as orientation-preserving self-diffeomorphisms of $\D$, preserving $n$-marked points on the $x$-axis and their associated vectors, all considered up to isotopy preserving said directions. We will prefer to work with an alternate description of $B_{*n}$ by considering the orientation-preserving self-diffeomorphisms of $\D$ with $n$ disjoint disks along the $x$-axis removed. Here we additionally choose a marked point on each interior boundary component and then require the self-diffeomorphisms to fix the outer boundary pointwise and preserve the interior boundaries along with their marked points. (Again this is considered up to isotopies which fix the outer boundary, preserve the interior boundaries, and fix the marked points.) Similar to $B_n$, $B_{*n}$ can be generated by half-twists along arcs connecting two boundary components and Dehn twists about curves parallel to the interior boundary components. 

Lastly, for each of the groups $\Gamma_{g,r}^s$, $B_n$ and $B_{*n}$, we use the functional notation induced by a composition of maps, acting from right to left. Compatibly, given a braid picture, we will read its action starting from the right. We will frequently work with representatives of group elements while presenting mapping group relations, which should be understood to hold up to isotopies of these elements in the actual group.

\subsection{Lefschetz fibrations and symplectic $4$-manifolds} \

A \dfn{Lefschetz fibration} is a surjective map $f\colon\, X\to \Sigma$, where $X$ and $\Sigma$ are $4$- and $2$-dimensional compact manifolds, respectively, such that $f$ fails to be a submersion along a non-empty discrete set $C$, and around each \textit{critical point} in $C$ it conforms to the local model $f(z_1,z_2)=z_1 z_2$, compatible with orientations. When $\partial X \neq \emptyset$, we assume that $C$ lies in the interior of $X$. We will moreover assume that each \emph{singular fiber} contains only one critical point, which can always be achieved after a small perturbation of any given Lefschetz fibration. Lastly, all the Lefschetz fibrations we consider will be over $\Sigma=\tS$, unless explicitly stated otherwise. A \dfn{Lefschetz pencil} is then defined as a Lefschetz fibration on the complement of a discrete set $B$ in $X$ (for $\partial X = \emptyset$), where $f$, around each \textit{base point} in $B$ conforms to the local model $f(z_1,z_2)=z_1/z_2$. We will often denote a Lefschetz fibration or a Lefschetz pencil as a pair $(X,f)$. 

It turns out that the monodromy of a Lefschetz fibration $f\colon X\to \D$ over the disk with a single
critical point is a positive Dehn twist along an embedded simple closed curve on a reference regular fiber, called the \emph{vanishing cycle} of this critical point. The critical point arises from attaching a $2$--handle, called the \dfn{Lefschetz handle} to the regular fiber with framing $-1$ with respect to the framing induced by the fiber. The attaching circle of this handle is the vanishing cycle, which contracts to the corresponding singular point. 

It follows that the monodromy of a Lefschetz fibration $f\colon X \to \tS$ with $n$ critical points is given by a factorization of the identity element $1\in \Gamma _g$ as  
\begin{equation} \label{monodromyfactorization}
1=\prod _{i=1}^n t_{v_i} \, ,
\end{equation}
where $v_i$ are the vanishing cycles of the singular fibers and $t_{v_i}$ is the positive Dehn twist about $v_i$ for $i=1, \ldots, n$. This factorization of the identity is called the \emph{monodromy representation} or the \emph{monodromy factorization} of $f$. Conversely, any word
\[ 
w=\prod _{i=1}^n t_{v_i} \]
prescribes a Lefschetz fibration over $\D$, and if $w=1$ in $\Gamma_g$ we get a Lefschetz fibration $X\to \tS$.

For a Lefschetz fibration $f: X\to \tS$, a \emph{section} is a map $\sigma  \colon \tS  \to X$ so that $f \circ \sigma = id_{\tS}$. Suppose that a fibration $f\colon X\to \tS$ admits a section $\sigma$ such that $S = \sigma(\Sigma) \subset X \setminus C$. Then the section $S$ provides a lift of the monodromy representation $\pi_1 (\Sigma \setminus f(C)) \to \Gamma_g$ to the mapping class group $\Gamma _{g,1}$. One can then fix a disk neighborhood of this section preserved under the monodromy, and get a lift of the factorization to $\Gamma_g^1$, which equals to a power of the boundary parallel Dehn twist. That is we get a defining word
  \[
  t_\delta ^m=\prod _i t_{v'_i} 
   \]
in $\Gamma_g^1$, where $v'_i$ are lifts of $v_i$ to $\Gamma_g^1$, $\delta$ is a boundary parallel curve, and $m$ is the negative of the self-intersection number of the section $S$. Conversely, any word in $\Gamma_g^1$ as above prescribes a genus $g$ Lefschetz fibration with vanishing cycles $v_i$ and with a distinguished section $S$ of self-intersection $-m$. These observations generalize in a straightforward fashion when we have $r$ \emph{disjoint} sections $S_1, \ldots, S_r \subset X \setminus C$, corresponding to $r$ marked points captured in the mapping class group $\Gamma_{g,r}$. 

A well-known fact is that one can blow-up all the base points of a genus $g$ Lefschetz pencil $(X,f)$ and obtain a genus $g$ Lefschetz fibration $(\tilde{X}, \tilde{f})$ with $r$ disjoint \emph{$-1$ sections}, that is sections of self-intersection number $-1$, each corresponding to a base point in $B$. In this case one obtains a relation
\[
  t_{\delta_1} \cdot \ldots \cdot t_{\delta_r} =\prod _i t_{v'_i} \, 
   \]
in $\Gamma^r_g$ (once again lifted from a factorization of the identity in $\Gamma_{g,r}$.). Conversely, whenever we have such a relation, we can construct a genus $g$ Lefschetz fibration with $r$ disjoint $-1$ sections, which can be then blown-down to obtain a genus $g$ Lefschetz pencil on a closed $4$-manifold.

A Lefschetz fibration (resp. pencil) $(X, f)$ and a symplectic form $\omega$ on $X$ are said to be \emph{compatible} if all the fibers are symplectic with respect to $\omega$ away from the critical points (resp. critical points and base points). The following theorem, coupled with the above discussion, gives a combinatorial way to study symplectic $4$-manifolds:

\begin{theorem}[Donaldson \cite{Donaldson}, Gompf \cite{GS}] 
Any symplectic $4$-manifold $(X, \omega)$ admits a compatible Lefschetz pencil, with respect to $\omega$. Conversely, any Lefschetz fibration $(X,f)$ with a homologically essential fiber admits a compatible symplectic form $\omega$ with respect to which any prescribed collection of disjoint sections are symplectic. 
\end{theorem}

\subsection{Open books and contact $3$-manifolds} \

An \dfn{open book decomposition} $\mathcal{B}$ of a $3$--manifold $Y$ is a pair $(L,f)$ where $L$ is an oriented link in $Y$, called the \dfn{binding}, and $f\co Y \setminus L \to S^1$ is a fibration such that  $f^{-1}(t)$ is the interior of a compact oriented surface $F_t \subset Y$ and $\partial F_t=L$ for all $t \in S^1$. At times we will denote this open book by the pair $(Y, f)$. The surface $F=F_t$, for any $t$, is called the \emph{page} of the open book. The \emph{monodromy} of an open book is given by the return map of a flow transverse to the pages and meridional near the binding, which is an element $\mu \in \Gamma_{g}^s$, where $g$ is the genus of the page $F$, and $s$ is the number of components of $L =\partial F$. 

Suppose we have a Lefschetz fibration $f\co X \to \D$ with bounded regular fiber $F$, and let $p$ be a regular value in the interior of the base $\D$. Composing $f$ with the radial projection $\D \setminus \{p\} \to \partial \D$ we obtain an open book decomposition on $\partial X$ with binding $\partial f^{-1}(p)$. Identifying $f^{-1}(p) \cong F$,
we can write 
\[\partial X=(\partial F\times \D)\cup f^{-1}(\partial \D) \, .\]
Thus we view $\partial F\times \D$ as the tubular neighborhood of the binding $L=\partial f^{-1}(p)$, and the fibers over $\partial \D$ as its \dfn{truncated pages}. The monodromy of this open book is prescribed by that of the fibration. Any open book whose monodromy can be written as a product of positive Dehn twists can be filled by a Lefschetz fibration over the $2$-disk. In this case, we say that the open book $(L, f|_{\partial X \setminus L})$ \dfn{bounds}, or is \dfn{induced by}, the Lefschetz fibration $(X,f)$. 

There is an elementary modification of these structures: Let $f\co X \to \D$ be a Lefschetz fibration with bounded regular fiber $F$. Attach a $1$--handle to $\partial F$ to obtain $F'$, and then attach a Lefschetz $2$--handle along an embedded loop in $F'$ that goes over the new $1$--handle exactly once. This is called a \dfn{positive stabilization} of $f$. A  Lefschetz handle is attached with framing $-1$ with respect to the fiber, and therefore it introduces a positive Dehn twist on $F'$. If the focus is on the $3$--manifold, one can forget the bounding $4$--manifold and view all the handle attachments in the \linebreak $3$-manifold. Either way, stabilizations correspond to adding canceling handle pairs, so diffeomorphism types of the underlying $4-$ and $3$-manifolds do not change, whereas the Lefschetz fibration and the open book decomposition change in the obvious way. It turns out that stabilizations preserve more than the underlying topology, as we will discuss shortly.

A contact structure $\xi$ on a $3$--manifold $Y$ is said to be \dfn{supported by an open book} $\mathcal{B}=(L,f)$ if $\xi$ is isotopic to a contact structure given by a $1$--form $\alpha$ satisfying $\alpha>0$ on positively oriented tangents to $L$ and $d\alpha$ is a positive volume form on every page. When this holds, we say that the open book $(L,f)$ is \dfn{compatible with the contact structure} $\xi$ on $Y$. Giroux proved the following remarkable theorem regarding compatibility of open books and contact structures:

\begin{theorem} [Giroux \cite{Gi2}] \label{Giroux}
Let $M$ be a closed oriented $3$--manifold. Then there is a one-to-one correspondence between oriented contact structures on $M$ up to isotopy and open book decompositions of $M$ up to positive stabilizations and isotopy.
\end{theorem}

We will call a Lefschetz fibration \dfn{allowable}, if all the fibers have non-empty boundaries, and if no fiber contains a \textit{closed} embedded surface. The following theorem brings all these structures together:

\enlargethispage{1cm}

\begin{theorem} [Loi--Piergallini \cite{LP}, also see Akbulut--Ozbagci \cite{AO}] \label{PALF}
An oriented compact $4$--manifold with boundary is a Stein surface, up to orientation-preserving diffeomorphisms, if and only if it admits an allowable Lefschetz fibration over the $2$-disk. Moreover, any two allowable Lefschetz fibrations over the $2$-disk bounded by the same open book carry Stein structures which fill the same contact structure induced by the boundary open book. 
\end{theorem}

\noindent As shown in \cite{BV}, one direction of the theorem generalizes to the case of allowable Lefschetz fibrations over arbitrary orientable surfaces with non-empty boundaries.

\newpage
\section{Arbitrarily long positive factorizations} \label{Sec: Construction}

In this section we prove Theorem~\ref{mainthm}. Here is the reader's guide to our proof:   
\begin{itemize}

\item In Sections~\ref{Sec: Garside} and \ref{Sec: Swap} we study a particular map of a surface of genus $2g+1$ with two boundary components which exchanges two genus $g$ subsurfaces. Using a branched cover construction we show that these maps, which we call \textit{swap maps}, can be realized as lifts of quasipositive braids. Propositions~\ref{prop:pos} and~\ref{prop:rhoproperties} sum up some properties of the swap map that are immediate from this very description. Swap maps will be the building blocks of the monodromies we will construct in the later subsections. 

\item In Section~\ref{Sec: Genus11} we picture a genus $11$ surface $F$ with two boundary components as a union of four copies of genus two surfaces with two boundary components, labeled $F_1, \ldots, F_4$, attached to two disks $D_1$ and $D_2$ with four holes; see Figure~\ref{fig:genus11}. Here we discuss a collection of swap maps as mapping classes acting on $F$, exchanging pairs $F_i$ and $F_j$, along with their interactions with mapping classes supported on these genus two subsurfaces (Proposition~\ref{prop:rho1}).

\item In Section~\ref{Sec: Examples} we review a particular subfamily of relations in $\Gamma_2^2$, which were discovered in \cite{BKM}, involving a single (varying) commutator and having unbounded length as products of positive Dehn twists. We will then show how to encode this family of commutators on a fixed genus $2$ subsurface of $\Sigma_{11}^2$ by forming a particular product of swap maps. This will give us the first example of a map $\Phi \in \Gamma_{11}^2$ with arbitrarily long positive factorizations; see Theorem~\ref{thm:minimal example}. We also observe that  the same construction hands us similar mapping classes in every $\Gamma_{g}^2$ with $g=11+4l$, $l \in \N$.

\item Lastly, in Sections~\ref{Sec: Framing} and~\ref{Sec: Extending}, we show how to extend the above factorizations of $\Phi$ to boundary multitwist in $\Sigma_{11}^2$. For this, we introduce a simple calculus of the swap maps acting on the genus $11$ surface $F$ as before, which includes the relations satisfied by the braid group; see Proposition~\ref{lm:rhoasbraid}. In particular, we see that it will be enough to understand how the swap maps identify different subsurfaces $F_i$ by keeping track of how they act as framed $4$-braids on the base surfaces $D_1$ and $D_2$ of $F \cong \Sigma_{11}^2$. We then extend these results to higher genera mapping classes, which completes the proof of our main theorem. 
\end{itemize}

\subsection{The Garside half-twist}\label{Sec: Garside} \
The \emph{Garside half-twist} is a braid $\Delta$ which looks like a half-twist placed in a ribbon: one aligns all the strands of a braid on a ribbon and twists the ribbon halfway around to the right. The resulting braid is shown in Figure \ref{fig:garside}. 

\begin{figure}[ht!]
\includegraphics[width = 5in]{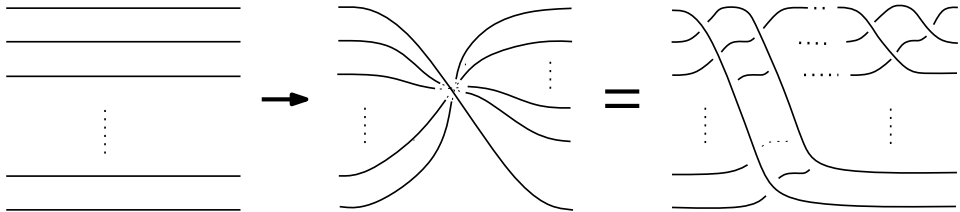}
\caption{The Garside half-twist $\Delta$ \label{fig:garside}. In the braid word description in Equation~\ref{eq:braid}, we are reading the braid from right to left.}
\end{figure}

The $n$-stranded Garside half-twist has a standard braid presentation with $b_i$ the braid half-twist that exchanges the $i$ and $i+1$ strands. Reading the braid from right to left, in $B_n$ we have

\begin{align} \label{eq:braid} \Delta &= (b_1 \cdot b_{2} \cdots b_{n-2} \cdot b_{n-1})(b_{1} \cdot b_2 \cdots b_{n-2}) \cdots (b_1 \cdot b_2)(b_1) . \end{align}

If we think of $\Delta$ as acting on the unit disk $\D$ with marked points $p_1, \dots, p_{n}$ sitting on the $x$-axis, then we have a geometric picture of the map: $\Delta$ is a rigid, 180$^\circ$ counter-clockwise rotation of the disk, followed by an isotopy supported near the boundary which slides the boundary back in the clockwise direction to where it started. See Figure \ref{fig:garside map}.

\begin{figure}[ht!]
\includegraphics[width = 3.8in]{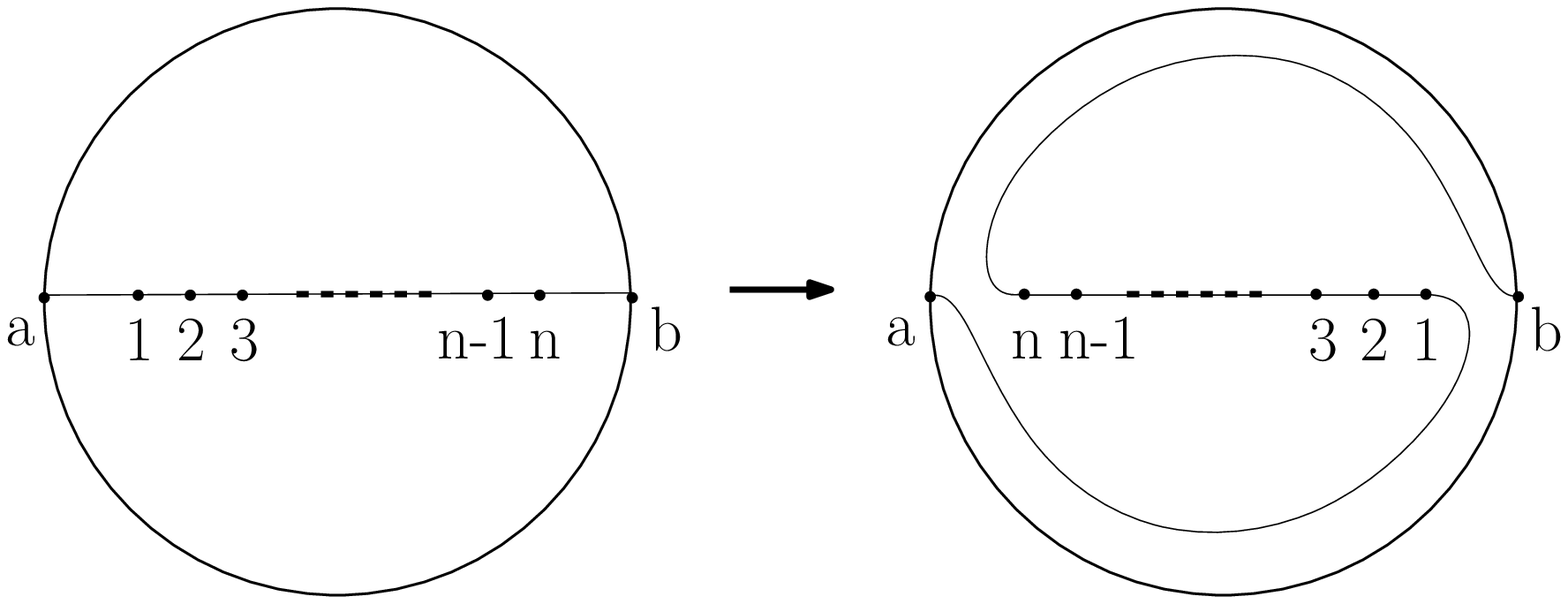}
\caption{ \label{fig:garside map}}
\end{figure}

We want to lift $\Delta$ along with its factorization and geometric description to the $2$-fold branched cover $S = \Sigma_g^2$ of the $2$-disk, to an involution $\bdel$ on $S$, so we assume $n= 2g+2$.

Geometrically, by aligning $S$ in $\R^3$ above $\D$ so that the branched cover involution on $S$ is a rotation about a line parallel to the $x$-axis, we can see $\bdel$ as a 180$^\circ$ right-handed rotation about the $z$-axis, again followed by a left-handed slide of the top boundary circle and a right-handed slide of the bottom boundary circle (as viewed and acted from above the $xy$-plane) back to where they started. (Because the orientations are opposed when we view the branched covering from the top or the bottom, these boundary rotations go in opposite directions, as shown in Figure \ref{fig:bdel lifted}.)

\begin{figure}[ht!] 
\begin{tabular}{cc}
\includegraphics[width = 2in]{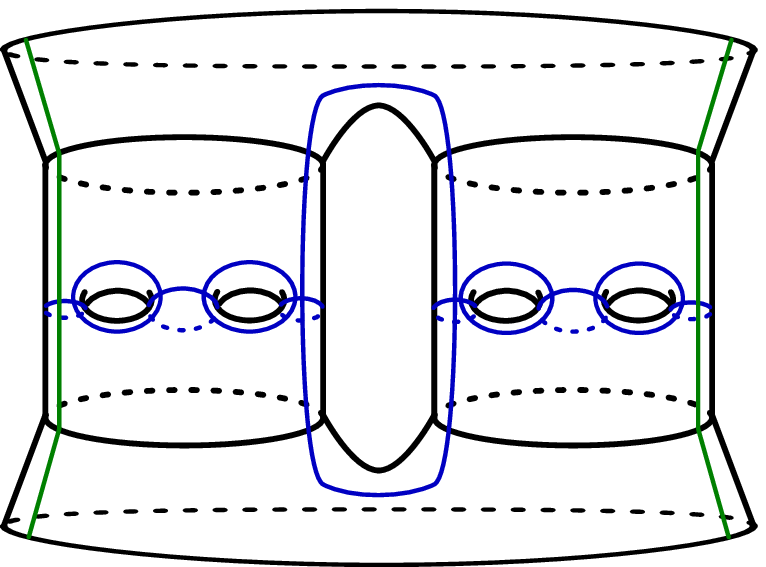} &
\includegraphics[width = 2in]{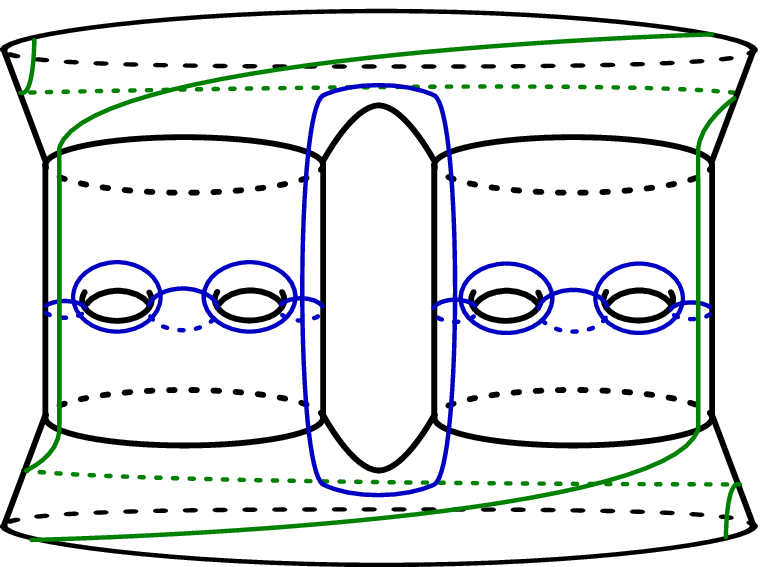}
\end{tabular}
\caption{The action of the Garside half-twist on the double branched covering surface $S = \Sigma_g^2$, illustrated for $g=5$. Notice the two disjoint subsurfaces of $S$ which are both diffeomorphic to $\Sigma_2^2$ and are 
exchanged under this action.} \label{fig:bdel lifted}
\end{figure}

$\bdel$ then also has a positive Dehn twist factorization by lifting the braid factorization of $\Delta$. If we let $t_i$ denote the lift of $b_i$ then 
\begin{align}
\bdel &= (t_1 t_2 \cdots t_{2g} t_{2g+1}) (t_1 t_2 \cdots t_{2g}) \cdots (t_1 t_2)(t_1) .
\end{align}
It is important to note that the square of $\Delta$ is a positive full braid twist, which is isotopic to a positive Dehn twist about a boundary parallel circle on $\D$ enclosing all the marked points. Because $n$ is even, this positive Dehn twist lifts to two positive Dehn twists in $S$, one about each boundary component. Calling this boundary multitwist $M_\bdry$, we have an equivalence
\[ \bdel^2 \simeq M_{\bdry} \]
up to isotopy in $S$ (fixing $\p S$).

\subsection{Swap maps} \label{Sec: Swap} \

From here on we assume that the genus $g$ of $S$ is odd and write $g = 2g'+1$. We will work with a decomposition of $S$ as two disjoint copies of $\Sigma_{g'}^2$, denoted $F_1$ and $F_2$, glued along two pairs of pants. (See Figure~\ref{fig:bdel lifted} for $g=5$.) As before, we will use diffeomorphisms and factorizations coming as lifts of braids and so we describe this surface as a double branched covering of the $2$-disk. Let $\D$ be the unit disk and arrange $n=2g+2=4g'+4$ marked points $p_1, \dots, p_n$ along the $x$-axis. Collect $ \{p_1, \dots, p_{2g'+2}\}$ in a smaller disk $R_1$ and $\{p_{2g'+3}, \dots, p_{4g'+4} \}$ in $R_2$ so that $D \setminus (R_1 \cup R_2)$ is a pair of pants. We can now see $S$ as the $2$-fold branched cover of $\D$ branched over $ \{p_1, \dots, p_{n}\}$, where $R_1$ and $R_2$ lift to the two disjoint genus $g'$ subsurfaces, $F_1$ and $F_2$ resp., and $\D \setminus (R_1 \cup R_2)$ lifts to the two disjoint pairs of pants connecting them. 

The \textit{swap map} is a slight\footnote{We use $\rho$ rather than $\bdel$ mainly because it facilitates our arguments in the last subsection. It also happens to have a shorter positive factorization than $\bdel$.} modification of $\bdel$. For $S$, $F_1$ and $F_2$ as above, let $M_i$ be the positive Dehn multitwist about the boundary of $F_i$, for $i=1,2$. The \textbf{swap map} is then an orientation-preserving self-diffeomorphism $\rho$ of $S$ defined by  
\begin{align} \rho \simeq& \bdel \cdot M_1^{-1} \cdot M_2^{-1} , \label{eqn:rho} \end{align}
which, under an identification of $S$ with $\Sigma_g^2$ descends to an element of the mapping class group $\Gamma_g^2$. Later we will impose some additional restrictions on the map $\rho$ as a diffeomorphism, though again, we never completely specify the diffeomorphism. These restrictions serve only to make our explanations easier. 

The definition of $\rho$ (and $\bdel$) now allows us to conclude: 

\begin{proposition} \label{prop:pos} 
The swap map $\rho$ can be factorized as a product of $g+1\;(= 2g'+2)$ positive Dehn twists.  
\end{proposition}

\begin{proof} Under the double branched covering, each quasipositive braid half-twist along an arc between two branched points lifts to a positive Dehn twist along a curve that covers the arc. The claim follows by observing that $\rho$ is a lift of a braid which admits a quasipositive factorization. The particular factorization is shown in Figure \ref{fig:rhoposfac}. 

\begin{figure}[htp!]
\labellist
\pinlabel {\Huge $\Delta$} at 85 90
\pinlabel {\LARGE $\Delta_1^{-2}$} at 170 138
\pinlabel {\LARGE $\Delta_2^{-2}$} at 170 40
\endlabellist
\begin{center}
\begin{tabular}{m{2in}m{.6cm}m{2in}}

\includegraphics[width = 1.9in]{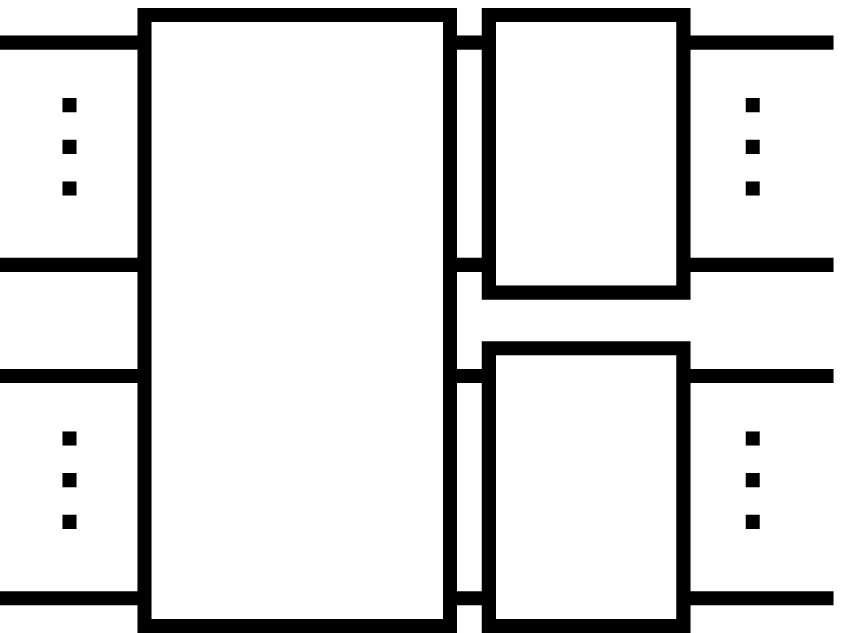} & 
\Huge{$\simeq$} &
\includegraphics[width = 1.9in]{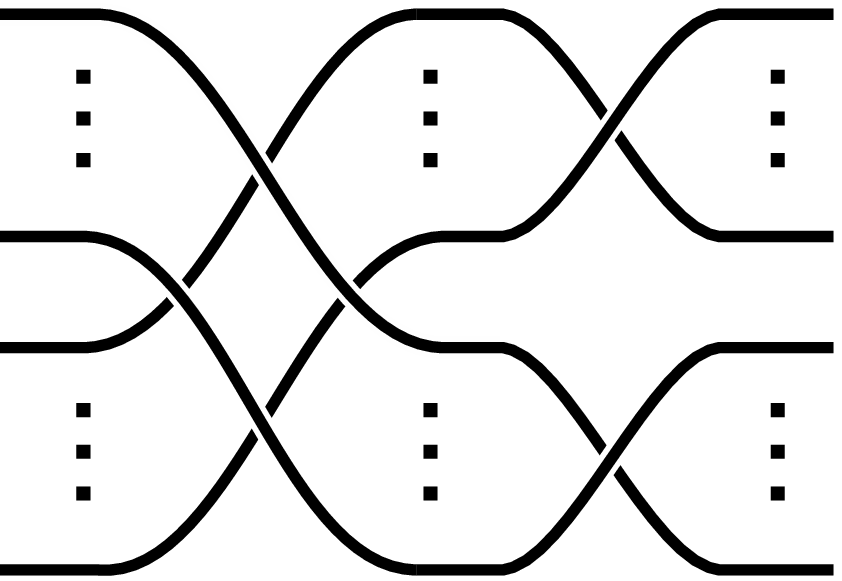} 
\end{tabular}
\vspace{12pt}
\Huge{ $\simeq $} \\
\vspace{6pt}
\includegraphics[width = 1.9in]{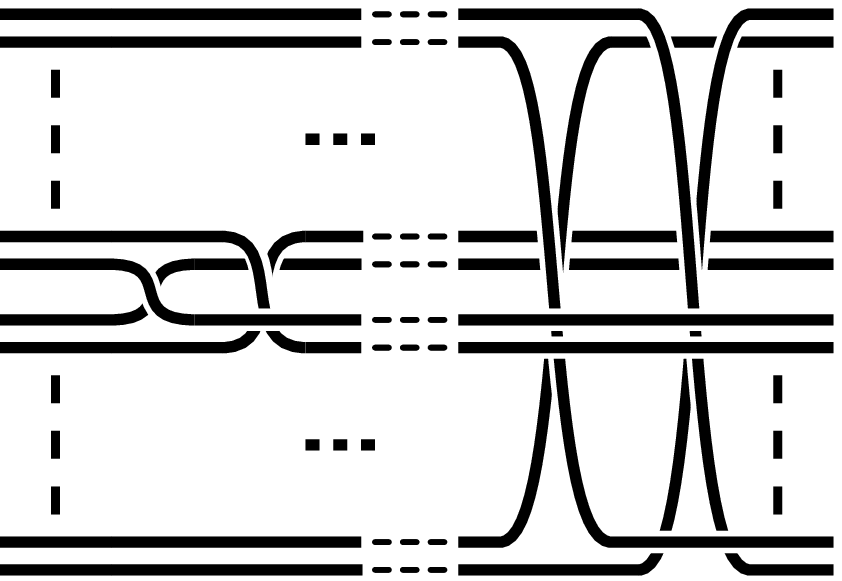}
\end{center}
\caption{A quasipositive factorization of the $4g'+4$ stranded braid $\Delta \cdot T_1^{-1} \cdot T_2^{-1}$, which lifts to $\rho$. Here $T_1$ is a full negative twist along the first $2g+2'$-strands in $R_1$ and $T_2$ in $R_2$. In the first diagram we see the obvious description of the braid $\Delta \cdot T_1^{-1} \cdot T_2^{-1}$, in the second we have a ribbon picture for it after canceling half-twists along $R_1$ and $R_2$. The final image gives the quasipositive factorization explicitly. \label{fig:rhoposfac}}
\end{figure}

\begin{figure}[htp!]%
\labellist
\pinlabel $c_1$ at 213 79
\pinlabel $c_2$ at 188 89
\pinlabel $c_3$ at 167 82
\pinlabel $c_4$ at 149 95
\pinlabel $c_5$ at 143 87 
\pinlabel $c_6$ at 118 89
\pinlabel $\gamma_3$ at 252 70
\endlabellist
\includegraphics[width=4in]{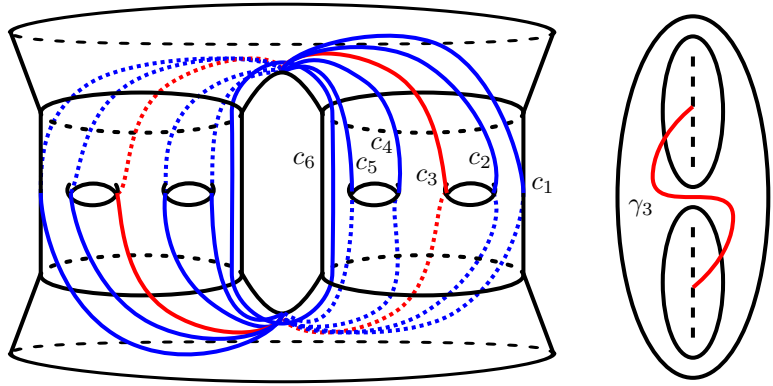}%
\caption{The left gives the positive factorization of $\rho$. The right describes a quasipositive half twist along the arc $\gamma_3$ from the factorization in Figure~\ref{fig:rhoposfac} which lifts to the Dehn twist along $c_3$ (in red) on the left. Using these Dehn twists $\rho = t_{c_6} \cdots t_{c_1}$.}%
\label{fig:rhoposfac2}%
\end{figure}
 
\end{proof}

Let us also note some properties of $\rho$, which are immediate from its geometric description:

\begin{proposition} \label{prop:rhoproperties} 
There exists a swap map $\rho$ on $S$, exchanging the subsurfaces $F_1$ and $F_2$ as shown above, so that:
\begin{enumerate}
\item $\ds \rho \big|_{F_2} \rho \big|_{F_1} = id \big|_{F_1}$ and $\ds \rho \big|_{F_1} \rho \big|_{F_2} = id \big|_{F_2}$
\item $\ds \rho^2 \big|_{F_1 \cup F_2} = id \big|_{F_1 \cup F_2}$ 
\item $\ds \rho^2 \simeq M_\p \cdot M_2^{-2} \cdot M_1 ^{-2}$
\end{enumerate}
\end{proposition}

\subsection{Swapping subsurfaces of a genus $11$ surface} \label{Sec: Genus11} \

Let $F$ be a surface diffeomorphic to $\Sigma_{11}^2$ and $F_1, \dots, F_4$ be disjoint subsurfaces in $F$ as shown in Figure \ref{fig:genus11}, each diffeomorphic to $\Sigma_2^2$. This decomposes $F$ into $F_1, \dots, F_4$ and two \textit{base surfaces} $D_1$ and $D_2$, each a $2$-disk with $4$ holes, such that each $F_i$ has one boundary component glued to the boundary of a hole in $D_1$ and the other boundary component to one in $D_2$. The outer boundaries of $D_1$ and $D_2$ make up the two boundary components of $F$. 

Additionally, there are four genus 5 subsurfaces $F_{ij}$, $1 \leq i < j \leq 4$, each diffeomorphic to $\Sigma_5^2$, obtained by taking the union of $F_i$ and $F_j$ and strips connecting the two along the arcs indicated in Figure~\ref{fig:genus11}. Each $F_{ij}$ is determined by a pair of holes and a proper arc $a_{ij}$ connecting them in each $D_i$. 

\begin{figure}[ht!] 
\includegraphics[width = \textwidth]{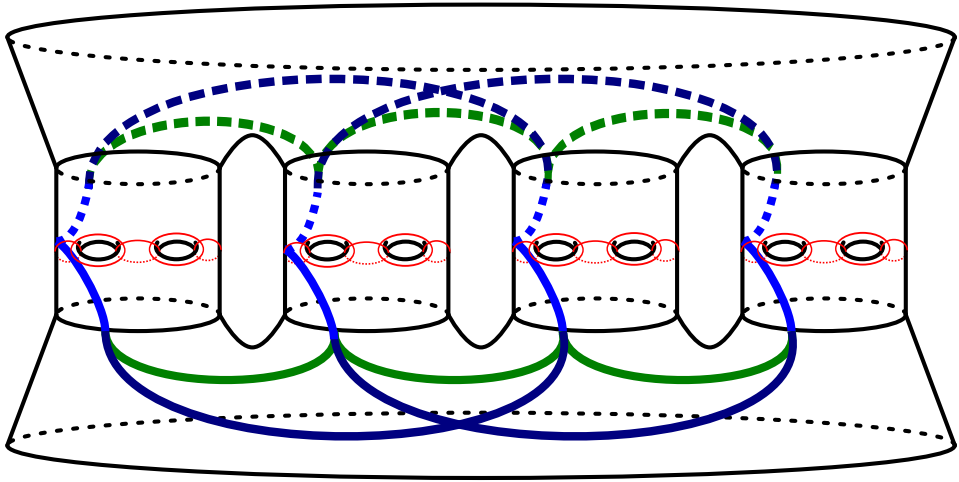} 
\caption{The genus $11$ surface $F$, containing genus $2$ subsurfaces $F_i$, $i=1, \dots, 4$ and genus $5$ subsurfaces $F_{ij}$,$1 \leq i < j \leq 4$, where $F, F_i, F_{ij}$ all have two boundary components. There is a collection of decorated circles on each $F_i$ which we will use to characterize the action of our swap maps on $F$. \label{fig:genus11}}
\end{figure}

In order to keep track of these identifications, we choose a chain of circles contained in $F_i$, as well as an arc, properly embedded in $F_i$, which connects the two boundary components through the one end component of the chain. We denote these arcs in $F_i$ by $a^i$, and the chain of circles by $c^i_1, \dots, c^i_{5}$. These arcs and circles are shown in Figure~\ref{fig:genus11}, along with the strips which connect $F_i$ and $F_j$. As this gives an arc decomposition of $F_i$, any diffeomorphism of $F_i$ is determined by where it sends this collection. We choose these arcs and circles by considering $F_i$ as the double branched cover of $\D$ as before, and then lifting the arcs that lay on the same axis as the branched points in $\D$ along with the arc connecting $p_1$ to the edge of the disk as in Figure~\ref{fig:garside map}. Moreover, the two chains in $F_i$ and $F_j$ along with the proper arcs $a^i$, $a^j$, and $a_{ij}$ in $F_{ij}$ give a chain which identifies $F_{ij}$ with our model surface $S$ from the previous subsection. This identification sends $F_i$ to $F_1$ and $F_j$ to $F_2$ and sends the chains indicated in Figure~\ref{fig:genus11} to the chain shown in Figure~\ref{fig:bdel lifted}. 

Using the chains to identify $F_{ij}$ with the surface $S$ translates the diffeomorphisms $\bdel$ and $\rho$ to a map on $F_{ij}$ which \emph{swaps} the two subsurfaces $F_i$ and $F_j$. Let $\bdel_{ij}$ and $\rho_{ij}$ denote these maps. By construction, each of $\bdel_{ij}$ and $\rho_{ij}$ behaves nicely with respect to the chains in $F_i$ and $F_j$. In particular, the identifications of subsurfaces $F_{ij}$ with the subsurfaces $F_1$ and $F_2$ of $S$ agree for $F_i$ and $F_j$ for all $i, j$. Thus, for example: 

\begin{itemize}
\item $\rho_{ij}$ and $\bdel_{ij}$ send the curve $c^i_k$ to $c^j_k$ and similarly $c^j_k$ to $c^i_k$ 
\item $\rho_{23} \rho_{12}$ sends the subsurface $F_1$ to the subsurface $F_3$ in the same way $\rho_{13}$ does. That is, 
\end{itemize}

\begin{align}  \rho_{23} \rho_{12} \big|_{F_1} = & \rho_{13}\big|_{F_1} . \end{align}

Under the above identifications of subsurfaces $F_i$, for any mapping class element $\ds A \in \Gamma_2^2$, let us write $A_i$ for the mapping class element in $\Gamma_{11}^2$ acting as $A$ on the subsurface $F_i$ and as identity on the rest of $F$. We then denote by $\rho_{ij}^A$, the map $A_i \rho_{ij} A_i^{-1}$. The map $\rho_{ij}^A$ also exchanges $F_i$ and $F_j$ but it changes the identifications: $\rho_{ij}^A$ maps $F_i$ to $F_j$ under the map $A^{-1}$ instead.

\begin{proposition} \label{prop:rho1}
For any $\ds A \in \Gamma_2^2$, $\rho_{ij}$ satisfy the following relations in $\Gamma_{11}^2$
\begin{enumerate}
\item $A_i \rho_{ij} = \rho_{ij} A_j$, $A_j \rho_{ij} = \rho_{ij} A_i$
\item $A_i \rho_{ij} A_i^{-1} = \rho_{ij}^A = A_j^{-1} \rho_{ij} A_j $
\end{enumerate}
\end{proposition}

\begin{proof} These properties are obvious from the standard conjugation action of mapping class elements. \end{proof}

\subsection{Arbitrarily long positive factorizations: simple examples } \label{Sec: Examples} \

The goal of this section is to construct the desired monodromies.

\begin{figure}[ht!]
\begin{center}
\includegraphics[width=9cm]{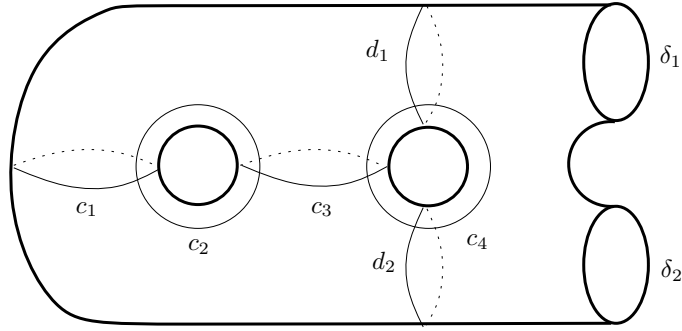}
\caption{The curves of the commutator relation.}
\label{CommutatorRelationCurves}
\end{center}
\end{figure}

\begin{lemma} \label{lemma:commutator}
Let $c_1, c_2, c_3, d_1, d_2$ be the simple closed curves on $\Sigma_2^2$  shown in Figure~\ref{CommutatorRelationCurves}, where $\delta_1, \delta_2$ are the two boundary curves. For any positive integer $m$, the following relation holds in $\Gamma_2^2$:
\[ 1 = T^m \, C(m) \,  , \]
where 
\[ T= t_{c_2} t_{c_3} (t_{c_1} t_{c_2} t_{c_3})^2 t_{c_1} t_{c_2} \] 
is a product of $10$ positive Dehn twists, and 
\[C(m)= [t_{c_1}^{-m} t_{d_1}^{m}, \psi^{-1}] \]  
is a commutator such that $\psi$ is any orientation-preserving self-diffeomorphism of $\Sigma_2^2$ compactly supported in the interior of $\partial \Sigma_2^2$ and mapping the pair $(c_1, d_1)$ to $(d_2, c_3)$. In particular, the above relation does not involve any twists along $\delta_1$ or $\delta_2$. 
\end{lemma}

The above family of relations is a slight modification of one of the many discovered in \cite{BKM}, which express fixed number of commutators and fixed powers of boundary parallel Dehn twists as arbitrarily long products of positive Dehn twists. In a way, we have chosen the simplest one: these relations consist of only one commutator and are supported in a genus two surface --- realizing the smallest possible numbers (of the power of a boundary parallel Dehn twist, fiber genus and base genus, respectively) for such a relation to hold \cite{BKM}. Moreover, it does not involve any Dehn twists along $\delta_1$ and $\delta_2$, allowing us to extend the mapping classes supported on these pieces in a rather straightforward way. Lastly, the proof of this particular relation is the easiest among all the others in \cite{BKM}. We include the proof below.

\begin{proof}
The standard chain relation on $\Sigma_2^2$ applied to $c_1, c_2, c_3$ gives
\[ \left( t_{c_1}t_{c_2}t_{c_3} \right)^{4}=t_{d_l} t_{d_2} \, .\]
Multiplying both sides with $t_{c_1}^{-1}$ from the left and with $t_{c_3}^{-1}$ from the right we get:
\begin{eqnarray}\label{eqn:T-l}
\left( t_{c_2}  t_{c_{3}} \right) \left( t_{c_1}t_{c_2} t_{c_{3}} \right)^{2} \left( t_{c_1}t_{c_2} \right) = t_{c_1}^{-1} t_{d_1} t_{d_2}  t_{c_{3}}^{-1}.
\end{eqnarray}
Let $T$ denote the left hand side of the above equation, and take the $n$-th power of both sides. 

Since $c_1, c_3, d_1, d_2$ are all pairwise disjoint, Dehn twists along them all commute with each other, allowing us to rewrite the $n$-th power of the right hand side as
\[ t_{c_1}^{-m}  t_{d_1}^{m} t_{d_2}^{m}   t_{c_{3}}^{-m} = t_{c_1}^{-m} t_{d_1}^{m} t_{\psi(c_1)}^{m}   t_{\psi(d_1)}^{m}
= t_{c_1}^{-m} t_{d_1}^{m} \psi( (t_{c_1}^{-m} t_{d_1}^{m})^{-1}) \psi^{-1} = [ \psi^{-1}, t_{c_1}^{-m} t_{d_1}^{m}] , \]
where $\psi$ is any self-diffeomorphism of $\Sigma_2^2$ as described in the statement of the lemma. By the Alexander method, it is easy to see that such $\psi$ exists. 
\end{proof}

We can now prove:

\begin{theorem} \label{thm:minimal example} 
Let $\ds \Phi = \rho_{24} \rho_{13} \rho_{34}\rho_{23}\rho_{12} \in \Gamma_{11}^2$. For each $m \in \N$, $\ds \Phi$ can be written as a product of $10m + 30$ positive Dehn twists. Moreover, for each $l \in \N$ there is an element $\Phi_l \in \Gamma_g^2$, $g=11+4l$, which can be written as a product of $10m + 5(2l+6)$ positive Dehn twists.
\end{theorem}

\begin{proof} 
Let us consider the surface $F \cong \Sigma_{11}^2$ with the four subsurfaces $F_i$ as before. Using Proposition~\ref{prop:rho1} repeatedly, we conclude that the following relation holds in $\Gamma_{11}^2$:
\begin{align*} & (B_1^{-1} A_1^{-1} B_1)(B_2^{-1} A_2^{-1})(B_3^{-1}) (\rho_{24}^B \rho_{13}^A \rho_{34}\rho_{23}\rho_{12})(B_3)(A_2 B_2)(B_1^{-1} A_1 B_1) \\ =& (B_1^{-1} A_1^{-1} B_1 A_1) \rho_{24} \rho_{13} \rho_{34}\rho_{23}\rho_{12}  \\ =& [A_1, B_1] \rho_{24} \rho_{13} \rho_{34}\rho_{23}\rho_{12} 
\end{align*}

Now fix a non-negative integer $m$ and let $A := t_{c_1}^{-m} t_{d_1}^{m}$, $B := \psi^{-1}$, and $T$ be as in Lemma~\ref{lemma:commutator}. The following gives a factorization of $\Phi$:

\begin{align*} & T_1^m (B_1^{-1} A_1^{-1} B_1)(B_2^{-1} A_2^{-1})(B_3^{-1}) (\rho_{24}^B \rho_{13}^A \rho_{34}\rho_{23}\rho_{12})(B_3)(A_2 B_2)(B_1^{-1} A_1 B_1) \\ &\; =  T_1^m [A_1, B_1] \rho_{24} \rho_{13} \rho_{34}\rho_{23}\rho_{12} \\ &\; =  T_1^m C(m)_1 \rho_{24} \rho_{13} \rho_{34}\rho_{23}\rho_{12} \\ & \; = \rho_{24} \rho_{13} \rho_{34}\rho_{23}\rho_{12} \\ & \; = \Phi
\end{align*}

The word \[T_1^m (B_1^{-1} A_1^{-1} B_1)(B_2^{-1} A_2^{-1})(B_3^{-1}) (\rho_{24}^B \rho_{13}^A \rho_{34}\rho_{23}\rho_{12})(B_3)(A_2 B_2)(B_1^{-1} A_1 B_1)\] is itself a positive factorization of $\Phi$; $T$ is already a product of positive Dehn twists, whereas each $\rho_{ij}$ has a positive factorization by Proposition~\ref{prop:pos}, and hence so do all its conjugates. 

By Equation~\ref{eqn:rho} and Proposition~\ref{prop:pos}, each instance of $\rho_{ij}$ can be written as a product of $6$ positive Dehn twists. By Lemma~\ref{lemma:commutator}, $T$ is a product of $10$ positive Dehn twists. Thus 
\[T^m (B_1^{-1} A_1^{-1} B_1)(B_2^{-1} A_2^{-1})(B_3^{-1}) (\rho_{24}^B \rho_{13}^A \rho_{34}\rho_{23}\rho_{12})(B_3)(A_2 B_2)(B_1^{-1} A_1 B_1)\]
gives a positive Dehn twist factorization of $\Phi$ of the stated length, concluding our proof of the first statement.

To obtain $\Phi_l \in \Gamma_{11+4l}^2$, we first replace each $F_i$ with a surface $F'_i \cong \Sigma_{2+l}^2$ connecting an inner hole of $D_1$ to that of $D_2$ so as to get a surface $F' \cong \Sigma_{11+4l}^2$ with subsurfaces $F'_i \cong \Sigma_{2+l}^2$, for $i=1, \ldots, 4$. Note that extending $\Sigma_2^2$ by attaching a torus with two holes, we obtain the same family of relations in $\Gamma_{2+l}^2$ as in Lemma~\ref{lemma:commutator} involving the same mapping classes, now understood to be mapping classes acting on $\Sigma_{2+l}^2$ which restrict to identity over the torus. (The chains on these subsurfaces are extended in the obvious way as well.) We can then define the swap maps $\rho'_{ij}$ in a similar fashion, still guaranteeing that while exchanging any pair of subsurfaces $F'_i$ and $F'_j$, they exchange the subsurfaces $F_i$ and $F_j$ as well. Each $\rho'_{ij}$ can be written as a product of $2l + 6$ positive Dehn twists. Hence, we can run the above proof mutatis mutandis to produce the positive factorizations of $\Phi_l$. 
\end{proof}

\subsection{Swap maps via framed braids} \label{Sec: Framing} \

Recall that $B_{*n}$ denotes the framed $n$-stranded braid group, which we think of as the mapping class group of a disk with $n$ holes, each with a marked point on the boundary, where we require the maps and isotopies to preserve the set of interior boundaries and their markings. There is a standard splitting of $B_{*n}$ as $B_n \x Z^n$. We can associate an $n$-vector of integral \textit{framings} to any framed braid, $b$: allow an ambient isotopy of $\D$, fixing $\bdry \D$, sending $b$ back to the identity, which then moves the interior boundary circles back to where they were before the action of $b$. For each boundary circle, we then count the number of clockwise rotations that boundary makes in its movement under this isotopy. One could just as easily count the counterclockwise rotations made by $b$ itself considered as the above movie with time reversed.  In addition to absolute framings as integers, we will often keep track of this information by using arcs connecting each interior boundary circle to the boundary of $\D$; see for instance Figures~\ref{fig:deltabraid} and~\ref{fig:rhobraid}. 

\begin{figure}[ht!] 
\labellist
\pinlabel $F_1$ at 6 137
\pinlabel $a^1$ at 58 82
\pinlabel $c_1^1$ at 18 113
\pinlabel $c_2^1$ at 46 129
\pinlabel $c_3^1$ at 64 105
\pinlabel $c_4^1$ at 81 129
\pinlabel $c_5^1$ at 113 113
\pinlabel $f^1$ at 80 216
\pinlabel $F_2$ at 231 141
\pinlabel $a_{12}$ at 130 60

\pinlabel $a^2$ at 170 82
\pinlabel $c_1^2$ at 130 113
\pinlabel $c_2^2$ at 158 129
\pinlabel $c_3^2$ at 176 105
\pinlabel $c_4^2$ at 193 129
\pinlabel $c_5^2$ at 225 113
\pinlabel $f^2$ at 192 221

\endlabellist
\includegraphics[width = 3in]{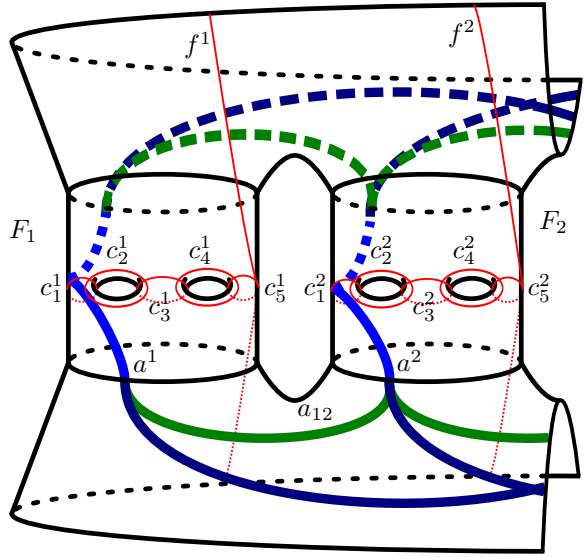} 
\caption{Half of the genus $11$ surface $F$; the genus $5$ subsurface $F_{12}$ and subsurfaces $F_1$ and $F_2$ with their chains and framing arcs. \label{fig:genus11-half}}
\end{figure}

Once again, consider the genus $11$ surface $F$ and its genus $2$ subsurfaces $F_1, \dots, F_4$ as in Figure \ref{fig:genus11}, and the four genus $5$ subsurfaces $F_{ij}$, $1 \leq i < j \leq 4$ obtained by taking the union of $F_i$ and $F_j$ and strips connecting the two along the arcs indicated in Figure~\ref{fig:genus11}. 

Each $F_{ij}$ is determined by a pair of holes and a proper arc $a_{ij}$ connecting them in each $D_i$. To make subsequent descriptions simpler, we choose our arcs $a_{ij}$ to be symmetric under an orientation preserving identification of $D_1$ with $D_2$ which preserves the interior boundary identifications. (As drawn, this symmetry is induced by the hyperelliptic involution evident in Figure~\ref{fig:genus11}.)

As in Section~\ref{Sec: Genus11} we use a collection of arcs and circles to keep track of the surfaces $F_i$ and the actions of the swap and Garside maps on them. In addition to the previous collection, we choose a second arc $f^i$ in $F^i$ which connect the two boundary components through the other end link of the chain, $c_5^i$. We use $f^i$ to keep track of framing information and often write $f^i$ for both the arc in $F_i$ and its extension to a properly embedded arc in $F$. These arcs and circles are shown in Figure~\ref{fig:genus11-half} for the subsurface $F_{12}$. As before, we chose these arcs by considering $F_i$ as the double branched cover of $\D$, and then lifting the arcs that lay on the same axis as the branched points in $\D$ along with the arcs connecting marked points $p_1$ and $p_n$ to the edge of the disk as in Figure~\ref{fig:garside map}. The collection of all circles $c^i_k$, $i =1, \dots , 4$, $k = 1, \dots, 5$, and framing arcs $f_i$, $i=1, \dots, 4$ give an arc and circle decomposition of the surface $F$ ($F$ cut along this union is a pair of disks) and so any diffeomorphism of $F$ is determined by how it acts on this collection. 

In order to completely control the behavior of the swap and Garside maps, we choose $\bdel_{ij}$ and $\rho_{ij}$ so that they exchange both the connecting and framing arcs $a^i$ and $f^i$ with $a^j$ and $f^j$ in addition to exchanging the circles $c_k^*$. 
This characterizes how each $\bdel_{ij}$ and $\rho_{ij}$ act on the union of subsurfaces $F_1 \cup \cdots \cup F_4$. To specify the maps, then, we need only to determine how they act on the base surfaces $D_1$ and $D_2$, preserving the boundaries and fixing them pointwise up to permutation. Any such map is equivalent to a framed $4$-stranded braid. We give explicit descriptions of these braids in Figures~\ref{fig:deltabraid} and ~\ref{fig:rhobraid}.

We sum up this discussion in the next lemma, which one should perhaps treat as something closer to a mantra:

\begin{lemma} 
Any composition of the maps $\bdel_{ij}$ and $\rho_{ij}$ are determined by their restriction to $D_1$; that is, they are determined by their induced framed $4$-stranded braid.
\end{lemma}

\begin{proof} We characterized the maps $\bdel_{ij}$ and $\rho_{ij}$ on $F_1 \cup F_2 \cup F_3 \cup F_4$ by requiring them to exchange the arc and circle chains of $F_i$ and $F_j$ and to fix $F_k$ for $k \neq i, j$ pointwise. Since $\bdel_{ij}$ and $\rho_{ij}$ act the same on $D_2$ as on $D_1$, we can then focus solely on $D_1$.  This allows us to consider the rest of the map acting on $D_1$, possibly exchanging boundary circles. Since we require the maps reserve the framing arcs $f_i$ in $F_i$, they also preserve the markings on the interior boundaries in $D^1$, which are also the boundaries of $F_i$. Thus on the complement of $F_1, \dots, F_4$, $\bdel_{ij}$ and $\rho_{ij}$ act as framed $4$-braids.
\end{proof} 

With this lemma in hand, we can give visual descriptions of the maps $\bdel_{ij}$ and $\rho_{ij}$ as acting on $D_1$. Figures~\ref{fig:deltabraid} and $\ref{fig:rhobraid}$ describe the maps $\bdel_{ij}$ and $\rho_{ij}$ respectively by their action on $D_1$.

\begin{figure}[ht!]%
\includegraphics[width=4in]{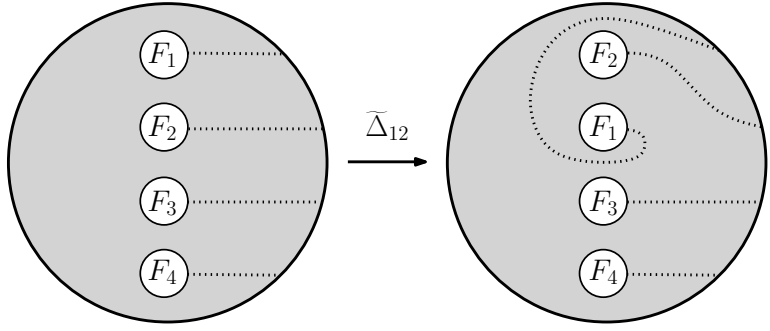}
\caption{The framed braid half-twist corresponding to $\bdel_{12}$ on $D_1$.}%
\label{fig:deltabraid}%
\end{figure}

\begin{figure}[ht!]%
\includegraphics[width=4in]{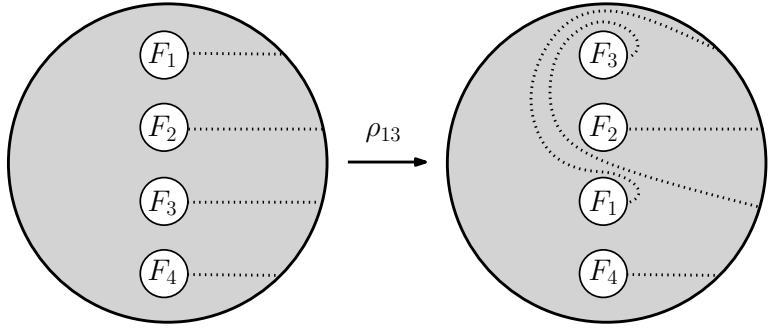}
\caption{The framed braid half-twist corresponding to $\rho_{13}$ on $D_1$. Note, this differs from the Garside half-twist $\bdel_{13}$ by \textit{negative} Dehn twists about the boundaries of $F_1$ and $F_3$ in $D_1$ --- this explains the framing difference between the two maps $\bdel_{ij}$ and $\rho_{ij}$.  }%
\label{fig:rhobraid}%
\end{figure}

Next proposition collects various properties of $\rho_{ij}$ acquired from well-known properties of braids, along with properties in Section~\ref{Sec: Genus11} regarding their interactions with self-diffeomorphisms of subsurfaces, amounting to a simple yet extremely helpful \emph{calculus} of swap maps:

\begin{proposition}[Properties of swap maps] \label{lm:rhoasbraid} 
For any $\ds A \in \Gamma_2^2$ the maps $\rho_{ij}$ satisfy the following relations in $\Gamma_{11}^2$
\begin{enumerate}
\item $\rho_{12}\rho_{23}\rho_{12} = \rho_{23}\rho_{12}\rho_{23}$
\item $\rho_{23}\rho_{34}\rho_{23} = \rho_{34}\rho_{23}\rho_{34}$
\item $\rho_{13} = \rho_{12}^{-1} \rho_{23} \rho_{12} =  \rho_{23} \rho_{12} \rho_{23}^{-1}$
\item $\rho_{24} = \rho_{23}^{-1} \rho_{34} \rho_{23} =  \rho_{34} \rho_{23} \rho_{34}^{-1}$
\item $\rho_{ij}\rho_{kl} = \rho_{kl}\rho_{ij}$ if $\{i, j \} \cap \{ k, l \} = \emptyset$ 
\item $A_i \rho_{ij} = \rho_{ij} A_j$, $A_j \rho_{ij} = \rho_{ij} A_i$
\item $A_i \rho_{ij} A_i^{-1} = \rho_{ij}^A = A_j^{-1} \rho_{ij} A_j $
\end{enumerate}
\end{proposition}

\begin{proof} 
For any one of the first five relations, observe that each diffeomorphism acts the same on the collection of parametrized subsurfaces $F_1, F_2, F_3, F_4$ (in fact with the same induced framed braid on the base disk $D_1$), so all follow from the properties of the associated $4$-stranded braid on $D_1$.
\end{proof}


\subsection{Completing the proof of Theorem~\ref{mainthm}} \label{Sec: Extending} \

We will now prove our main theorem, by extending our result in Theorem~\ref{thm:minimal example}. 

\begin{lemma} \label{lm:braidfulltwist} 
The maps $\rho_{ij}$ satisfy the following relations in $\Gamma_{11}^2$:
\[ (\rho_{34}\rho_{23}\rho_{12})^4 = M_\bdry M_4^{-4}M_3^{-4}M_2^{-4}M_1^{-4}\]
\end{lemma}

\begin{proof} 
Recalling that $\rho_{ij} = \bdel_{ij} M_j^{-1} M_i^{-1}$, we will proceed by checking the corresponding factorization for $\bdel_{ij}$. Given Proposition~\ref{lm:rhoasbraid}, it is easy to see that $(\bdel_{34}\bdel_{23}\bdel_{12})^4$ acts as a positive full twist on $D_1$. (This is just the usual factorization of a full twist on $4$-strands into the standard braid generators.) What remains is to determine how it acts on the framings of the inner boundary components.
 
To check the framings on the boundary components of $D_1$, we follow through the product as in Figures~\ref{fig:deltabraid} and \ref{fig:rhobraid}. Let us denote the marked inner boundary component of $D_1$ connected to $F_i$ by $\partial_i$, for $i=1, \ldots, 4$. As demonstrated in Figure~\ref{fig:deltabraid}, $\bdel_{ij}$ exchanges $\partial_i$ and $\partial_j$ followed by a full right-handed twist to $\partial_j$, where $i<j$. Thus $\bdel_{ij}$ sends $\partial_j$ to $\partial_i$ with framing $0$ and sends $\partial_i$ to $\partial_j$ with framing $+1$. In the composition $(\bdel_{34}\bdel_{23}\bdel_{12})^4$, each ``strand'' $\partial_i$ shows up as the top circle involved in a $\bdel$ map exchange exactly three times and as the bottom circle exactly three times as well. Thus $(\bdel_{34}\bdel_{23}\bdel_{12})^4$ acts as a full positive twist (along the outer boundary) and with framing $+3$ on each inner boundary component.

Now the boundary twist $M_\bdry$ acts by a full twist but with framing $+1$ on each inner boundary component. So 
\[(\bdel_{34}\bdel_{23}\bdel_{12})^4 = M_\bdry M_4^2 M_3^2 M_2^2 M_1^2 \]
in $\Gamma_{11}^2$, once again under a fixed identification $F \cong \Sigma_{11}^2$.

Since $\rho_{ij} = \bdel_{ij} M_j^{-1} M_i^{-1}$, we conclude that $(\rho_{34}\rho_{23}\rho_{12})^4$ also acts as a full braid twist but with framing $-3$ on each inner boundary component instead --- recall that each strand shows up six times in total. So in $\Gamma_{11}^2$ we have
\[(\rho_{34}\rho_{23}\rho_{12})^4 = M_\bdry M_4^{-4} M_3^{-4} M_2^{-4} M_1^{-4} .\]
\end{proof}

We also note the following simple fact:

\begin{lemma}[Inserting equals appending] \label{lm:insert}
Let $\Phi$ be any mapping class, and $W$ be a word in positive Dehn twists, so that $\Phi = W$ describes a positive factorization of $\Phi$. If $\Phi'$ is a mapping class with $\Phi'= W'$, where $W'$ is a word obtained by inserting positive Dehn twists into $W$, then there is a positive word $\tilde{W}$ so that appending $\tilde{W}$ to $W$ gives $W'$. That is, we can also write:
 $$W' = \tilde{W} W$$ 
\end{lemma}

\begin{proof} Suppose we split $W$ as $W_2 \cdot W_1$, a product of two words and that we insert a letter $w$ in between. Then $W_2 w W_1 = (W_2 w W_2^{-1})(W_2 \cdot W_1) =  \tilde{w} W$ in the mapping class group, where $\tilde{w}$ is again a positive Dehn twist.
\end{proof}

We can now complete the proof of our main theorem. 

\begin{proof}[Proof of Theorem \ref{mainthm}]
By Theorem~\ref{thm:minimal example} we know there are arbitrarily long factorizations of the map $\Phi$. To prove Theorem~\ref{mainthm}, we extend these factorizations to a factorization of the single positive Dehn multitwist, $M_{\bdry}$, about the boundary of the genus $11$ surface $F$. 

Recall that $\ds \Phi = \rho_{24} \rho_{13} \rho_{34}\rho_{23}\rho_{12}$, and given any $m$ there is a factorization of $\Phi$: 
\[ \ds \Phi = T_1^m (B_1^{-1} A_1^{-1} B_1)(B_2^{-1} A_2^{-1})(B_3^{-1}) (\rho_{24}^B \rho_{13}^A \rho_{34}\rho_{23}\rho_{12})(B_3)(A_2 B_2)(B_1^{-1} A_1 B_1) ,\]
where the right hand side can be expressed as a product of $10m+30$ positive Dehn twists. Here $A$, $B$, $T$ are mapping classes in $\Gamma_2^2$ coming from Lemma~\ref{lemma:commutator}. From Proposition~\ref{lm:rhoasbraid} we have $\rho_{13} = \rho_{12}^{-1} \rho_{23} \rho_{12}$ and $\rho_{24} = \rho_{23}^{-1} \rho_{34} \rho_{23}$. So 
\[ \Phi = \rho_{23}^{-1} \rho_{34} \rho_{23}\rho_{12}^{-1} \rho_{23} \rho_{12} \rho_{34}\rho_{23}\rho_{12}\]

Now we can insert copies of $\rho_{ij}$ into this factorization to get the word 
\[ (\rho_{34} \rho_{23} \rho_{12} \rho_{23}) \rho_{23}^{-1} \rho_{34} \rho_{23} (\rho_{12}^2) \rho_{12}^{-1} (\rho_{34}) \rho_{23} \rho_{12} \rho_{34}\rho_{23}\rho_{12}  = (\rho_{34}\rho_{23}\rho_{12})^4 . \]
Thus by Lemma~\ref{lm:insert}, there is some positive word $W$ so that 
\[M_\bdry M_4^{-4} M_3^{-4}M_2^{-4}M_1^{-4} = W \Phi\]
and hence:
\[M_\bdry = W \Phi M_4^{4} M_3^{4}M_2^{4}M_1^{4} \]
By replacing $\Phi$ by the aforementioned positive factorizations of length $10m+60$ we produce arbitrarily long positive factorizations of $M_\bdry$. 

Now, a similar argument shows how to extend these factorizations to the boundary twist on any higher genus surface. Let $F'$ be a genus $g$ surface ($g>11$) with two boundary components and let $F$ be a non-separating genus $11$ subsurface with two boundary components. Then there is a length $2g+1$ chain for $F'$ which contains a chain for $F$ as a subchain. In fact, we can assume the chain for $F$ consists of standard circles $c_1, \ldots, c_{23}$. Letting $t_i$ be the positive Dehn twist about $c_i$, the standard chain relations on $F$ and $F'$ are $M_\bdry = (t_1, \dots, t_{23})^{24}$ and $M'_\bdry = (t_1, \dots, t_{2g+1})^{2g+2}$, where $M_\bdry$ and $M'_\bdry$ are the boundary multitwists of $F$ and $F'$, resp. We can insert instances of $t_i$ into the word $(t_1, \dots, t_{23})^{24}$ to get $(t_1, \dots, t_{2g+1})^{2g+2}$ and so again by Lemma~\ref{lm:insert}, there is some positive word $W$ so that $M'_\bdry  = W M_\bdry$. By replacing $M_\bdry$ with the previous factorizations, we get factorizations of $M_\bdry'$ of unbounded length. 

\end{proof}

\begin{remark}\label{Genus8}
It is possible through a more involved argument to construct analogues of swap maps on a surface $\Sigma_4^1$, this time made of four copies of $\Sigma_2^1$ and a $2$-disk with four holes, exchanging different copies of $\Sigma_2^1$. An indirect construction of such a map can be seen in Section~7.4 of \cite{BEV}. This changes which framings can (or should) be associated to the new swap map for its interpretation via framed $4$-braids. It is then possible to obtain both Theorem~\ref{thm:minimal example} and Theorem~\ref{mainthm} for $g=8$ and higher, though the technical details are more complicated.
\end{remark}

\vspace{0.2in}

\section{Applications to symplectic and contact topology} \label{Sec: Applications}

We will now discuss various results on the topology of closed symplectic \linebreak $4$-manifolds and contact $3$-manifolds, all departing from Theorem~\ref{mainthm}. We begin by proving Theorems~\ref{cor1} and \ref{cor2} given in the introduction:

\begin{proof}[Proof of Theorem~\ref{cor1}]
Any relation of the form
\[ 
t_{\delta} = \text{A product of $m$ positive Dehn twists along non-separating curves}
\]
in the mapping class group $\Gamma_g^1$ gives rise to a relatively minimal genus $g$ Lefschetz fibration 
\[ (X(m),f(m))=(X_g(m),f_g(m)) \]
over the $2$-sphere with a section of self-intersection $-1$. By Theorem~\ref{mainthm} we have such a family of relations prescribing a family of relatively minimal Lefschetz fibrations $\{(X(m), f(m)) | m \in \N\}$. Since the Euler characteristics of a fixed genus $g$ Lefschetz fibration $(X(m),f(m))$ is given by
\[
\eu(X(m)) = 4-4g + m \, ,
\]
we see that the Euler characteristic of $X(m)$ is a strictly increasing in $m$. 

To prove the second claim we first take the mapping classes $\Phi_l \in \Gamma_g^2$ from Theorem~\ref{thm:minimal example}, where $g=11+4l$. It is easy to see that each one of these elements can be completed to boundary multitwists using more swap maps in the same exact way as in the proof of Theorem~\ref{mainthm}. 

As seen from the proof of Theorem~\ref{mainthm}, the monodromy factorization of each genus $11$ Lefschetz fibration $f(m)$ consists of a product of positive Dehn twists along $c_1, c_2, c_3$ supported in the subsurface $F_1$, various swap maps and their conjugates. As usual, we can calculate $H_1(X(m))$ by taking the quotient of $H_1(F)$, where $F$ is the regular (closed) fiber, by the normal group generated by the vanishing cycles. The key observation to make is that each swap map is an involution on the subsurface $F_{ij}$, whose action on the standard homology generators of $F_{ij}$ identifies the homology generators contained in the subsurface $F_i$ to that of $F_j$, while turning the two generators in the middle to $2$-torsion elements. Said differently, the vanishing cycles coming from a swap map $\rho_{ij}$ (See Figure~\ref{fig:rhoposfac2}) or its conjugates on $F_{ij}$ introduce relations in $H_1(F)$ which identify the homology generators supported in $F_i$ to those in $F_j$ while turning the two generators coming from the ``connecting genus'' in between the two into $2$-torsion elements. So we see that all infinite order homology generators of $F$ lie in $F_1$. On the other hand, the remaining vanishing cycles $c_1, c_2, c_3$ further kill three of these, while keeping the homology generator represented by $c_4$ in Figure~\ref{CommutatorRelationCurves} intact, thus allowing us to conclude that $b_1(X(m))=1$. It is now straightforward to generalize the same line of arguments to the genus $g=11+4l$ Lefschetz fibrations described above to conclude that for each $l$ we have $b_1=1+2l$. 

Blowing down one of the distinguished $-1$ sections of the above families of Lefschetz fibrations, we obtain the promised Lefschetz pencils in the theorem.
\end{proof}

\begin{proof}[Proof of Theorem~\ref{cor2}] 
Let $(Y_k), f_k)$ be the family of genus $g$ open books with two boundary components given by iterating $k$ times the monodromies we obtain in the previous section. Corresponding to the arbitrarily long positive factorization (parametrized by $m$) of each open book, there exists a Lefschetz fibration $(X_k(m), F_k(m))$ with regular fiber $F \cong \Sigma_g^2$ and base $D^2$. Since all the vanishing cycles are non-separating, each $(X_k(m), F_k(m))$ is an allowable Lefschetz fibration. Lastly, from the monodromy of the open book $f_k$ on $Y_k$ we easily see that $Y_k$ is the Seifert fibered $3$-manifold with base genus $g$ and two singular fibers of degree $k$, which implies that the $3$-manifolds $Y_k$, and therefore the family of contact \linebreak $3$-manifolds $\{ (Y_k, \xi_k) \, | k \geq 0 \}$ we obtain from the contact structures induced by the open books in hand are all distinct. Capping off one of the boundary components yields the promised families with connected binding.

Using the families of genus $g=11+4l$ Lefschetz fibrations we have in the proof of Theorem~\ref{cor1} above, we prove the second claim. The infinite families of contact \linebreak $3$-manifolds supported by genus $g=11+4l$ open books are the boundaries of allowable Lefschetz fibrations obtained by taking the $k$-times iterated (untwisted) fiber sums of these fibrations, from which we then remove a regular fiber and one of the distinguished sections of self-intersection $-k$. Thus the first Betti number of these bounding $4$-manifolds is the same as the first Betti number of one copy of them, that is, $1+2l$, which increases as we increase $l \in \N$. 
\end{proof}

\begin{remark}
It might be worth noting that taking twisted fibers sums, we can choose any non-negative integer to be the $b_1$ of all the Stein fillings in Theorem~\ref{cor2} for large enough $g$. It can be seen from the proof of Theorem~\ref{mainthm} that the chain of vanishing cycles we add result in symplectic $4$-manifolds with $b_1=0$ discussed in the \textit{first} part of Theorem~\ref{cor1}, once $g>11$. 
\end{remark}

\begin{remark}
All but finitely many members of any family of closed symplectic \linebreak $4$-manifolds or Stein fillings satisfying the statements of the above theorems can be seen to have $b^+>1$. This is immediate once we invoke the following theorem from \cite{Stipsicz2}: For any Stein filling $(X,J)$ of a fixed contact $3$-manifold $(Y, \xi)$, there is a lower bound for the sum $2 \eu(X) + 3 \sigma(X)$ which only depends on $(Y, \xi)$. That is $4- 4 b_1(X) + 5b^+(X) - b^-(X)$ is bounded from below, and since $|b_1(X)| \leq 2g-2$, large $\eu(X)$ implies large $b^+(X)$. (For closed symplectic $4$-manifolds, we take $(Y, \xi)$ to be the contact $3$-manifold induced by the boundary open book obtained after removing a regular fiber and a $(-1)$-section.)
\end{remark}

Thus, we obtain a new proof of our earlier result in \cite{BV}: 

\begin{corollary} \label{ALSFthm}
There are infinite families of contact $3$-manifolds, where each contact $3$-manifold admits a Stein filling with arbitrarily large Euler characteristic.
\end{corollary}

\begin{proof}
By Theorem~\ref{PALF} the total space of the allowable Lefschetz fibration \linebreak $(X_k(m),f_k(m))$ we used in the proof of Theorem~\ref{cor2} above admits a Stein structure inducing the same contact structure $\xi_k$ compatible with the open book on its boundary $Y_k$, independent of $m$. Varying $k$ we conclude the proof.
\end{proof}

\begin{remark}
This provides a new proof of the first part of our Theorem~1.1 in \cite{BV}, but for much restricted families: compare with Theorem~4.2 and the discussion at the end of Section~4.2 in \cite{BV}. However, we are unable to reproduce the second part of Theorem~1.1 which stated the existence of contact $3$-manifolds which admit Stein fillings not only with arbitrarily large Euler characteristics but also with arbitrarily small signatures. This is due to the significantly harder task of calculating the signatures of the filling from the given monodromies, when they are not seen to be hyperelliptic. One can in fact see that most of our monodromies are not hyperelliptic using the signature formula from \cite{Endo}: 
\begin{equation*}
\sigma(X)=-\frac{g+1}{2g+1}N+
\sum_{j=1}^{[\frac{g}{2}]}(\frac{4j(g-j)}{2g+1} -1)s_j .
\end{equation*}
Here $X$ is the total space of the hyperelliptic fibration (with closed fibers), $N$ and $s=\sum_{j=1}^{[\frac{g}{2}]}s_j$ are the numbers of nonseparating and separating vanishing cycles, respectively, whereas $s_j$ denotes the number of separating vanishing cycles which separate the surface into two subsurfaces of genera $j$ and $g-j$. Any one of our genus $g \geq 11$ Lefschetz fibration $(X,f)$ has $N= 10m + 104 + (2g+1)^{2g+2} - 23^{24}$ nonseparating vanishing cycles and no separating ones. We therefore see that we do not get an integer for most $g$ and $m$ allowing us to conclude that for these values, $(X,f)$ cannot be hyperelliptic. One can arrive at the same conclusion for our second family of genus $g=11+4l$ Lefschetz fibrations in the same way.
\end{remark}

Although we see that there are vast families of contact $3$-manifolds which appear as counter-examples to Stipsicz's Conjecture on the boundedness of the topology of Stein fillings \cite{Stipsicz2}, it is plausible that one can invoke Theorem~\ref{mainthm} to detect which ones can \emph{attain} this boundedness property. A simple invariant one can associate to any contact $3$-manifold can be described as follows: Any $\Phi \in \Gamma_g^1$ can be factorized by Dehn twists, which are \emph{all} positive except for the boundary parallel twists. This can be easily seen by looking at the homomorphism induced by capping off the boundary of $\Sigma_g^1$, whose kernel is generated by the boundary parallel Dehn twist  $t_\delta$. Let $\textbf{bt}(\Phi)$ denote the supremum of the powers of $t_\delta$ among these ``almost'' positive factorizations. Since any contact $3$-manifold can be supported by an open book of genus $g \geq 8$ with connected binding, we can now define the \textbf{boundary twisting} invariant $\textbf{bt}$ of a contact $3$-manifold $(Y, \xi)$ as the supremum of all $\textbf{bt}(\Phi)$ such that $\Phi$ is the monodromy of such an open book. It follows from \cite{Gi2} and \cite{LP} that
\[ \textbf{bt} : \{ \text{Stein fillable contact $3$-manifolds} \} \rightarrow \N \cup \{+ \infty \} . \]
On the other hand, by Theorem~\ref{mainthm} and the arguments applied in the proof of Corollary~\ref{ALSFthm} above, we conclude that
 
\begin{corollary} \label{invariant}
If there is an upper bound on the Euler characteristics of Stein fillings of a contact $3$-manifold $(Y, \xi)$, then  $\emph{\textbf{bt}}(Y, \xi)=0$. 
\end{corollary}

\vspace{0.2in}
\section{Final remarks} \label{Sec: Final} \
We finish with a few remarks and questions.

\noindent \textbf{Further constructions.} Our construction of the monodromies in the proofs of Theorems~~\ref{mainthm} and \ref{thm:minimal example} clings on expressing a single commutator on a genus $g \geq 2$ surface with boundary as arbitrarily long products of positive Dehn twists and no boundary parallel twists. As demonstrated in \cite{BV}, there are many other relations of this sort expressing a fixed number of commutators, say $h$ many, as arbitrarily long products of positive Dehn twists and $k$ boundary parallel Dehn twists for any $h \geq 1$ and $0 \leq k \leq 2h-2$. Using these relations and a larger variety of swap maps (i.e. exchanging more than four subsurfaces) instead, one can construct a larger collection of factorizations.

\noindent \textbf{Arbitrarily long positive factorizations in $\Gamma_g^n$.} Our method of constructing arbitrarily long positive factorizations requires at least four subsurfaces which are translated to each other in a commutator fashion and a relation supported in one of these subsurfaces expressing a commutator as an arbitrarily long product of positive Dehn twists. As discussed above, the latter requires the genus of each one of these subsurfaces to be at least two. Although it is possible to avoid the extra genera incorporated into our constructions as sketched in Remark~\ref{Genus8} above, we see that our current methods do not allow us to derive similar factorizations on a surface with genus less than eight. Finding other ways to produce arbitrarily long positive factorizations of mapping classes on surfaces of genus $g<8$ would be interesting. In general, it would be good to determine exactly for which pairs of integers $g$ and $n \geq 1$ is there an element in $\Gamma_g^n$ with arbitrarily long positive factorizations. 

\noindent \textbf{Lefschetz fibrations over surfaces of positive genera.} Generalizing Smith's question, one can ask for which fixed pairs of non-negative integers $g,h$ there exists an a priori upper bound on the Euler characteristic of a relatively minimal genus $g$ Lefschetz fibration over a genus $h$ surface \textit{admitting a maximal section}, i.e. a section of maximal possible self-intersection $-1$ if $h=0$ and of $2h-2$ if $h \geq 1$. When $h \geq 1$, the answer is completely determined by Korkmaz, Monden, and the first author, who proved that there is a bound if and only if $g=1$ \cite{BKM}. Combined with the results of this paper, this question now remains open only for the pairs $3 \leq g \leq 7$ and $h=0$.

\noindent \textbf{Stein fillings with bounded topology.} As per our discussion that led to Corollary~\ref{invariant} above, a curious question is the following: Does the vanishing of the boundary twisting invariant $\textbf{bt}$ (or any variant of it defined for smaller genus open books as well) determine precisely which Stein fillable contact $3$-manifolds have an a priori upper bound on the Euler characteristic of its Stein fillings? (As seen from Theorem~\ref{thm:minimal example}, there are mapping classes other than boundary multitwists which also admit arbitrarily long positive factorizations, suggesting that an affirmative answer to this question may not be very likely.)

\noindent \textbf{Underlying geometries.} In \cite{BV} we built arbitrarily large Stein fillings of contact structures on $3$-manifolds which were graph but not Seifert fibered, leading to the natural question on whether this aspect was related to the underlying geometry of the $3$-manifold. Our current work provides a negative answer: we see that there are Seifert fibered $3$-manifolds with one or two singular fibers, which can be equipped with contact structures admitting arbitrarily large Stein fillings. The property of admitting arbitrarily large Stein fillings therefore seems to be more connected to the underlying topology (such as how small the base can be for a Seifert fibered $3$-manifold) rather than the underlying geometry in general.

\end{document}